\documentclass[reqno]{amsart}
\usepackage{amssymb,
enumitem,
mathrsfs,
mathtools,
tikz,
upgreek,
verbatim,
hyphenat,
nicefrac
}

\usepackage[T1]{fontenc}
\usepackage[shadow]{todonotes}
\usepackage{setspace}
\definecolor{navyblue}{RGB}{0, 0, 128}

\newcommand{\NN}{\mathbb{N}}
\newcommand{\ZZ}{\mathbb{Z}}
\newcommand{\QQ}{\mathbb{Q}}

\newcommand{\TT}{\mathbb{T}}
\newcommand{\nc}{\newcommand}
\nc{\inv}{^{-1}}

\newcommand{\id}{\operatorname{id}}

\DeclareMathOperator{\Aut}{Aut}
\DeclareMathOperator{\tp}{tp}

\DeclareMathOperator{\Fix}{Fix}

\DeclareMathOperator{\po}{<_{po}}
\DeclareMathOperator{\npo}{\not<_{po}}

\newcommand{\indep}[1][]{%
  \mathrel{
    \mathop{
      \vcenter{
        \hbox{\oalign{\noalign{\kern-.3ex}\hfil$\vert$\hfil\cr
              \noalign{\kern-.7ex}
              $\smile$\cr\noalign{\kern-.3ex}}}
      }
    }\displaylimits_{#1}
  }
}

\newenvironment{enumerate-(a)}{\begin{enumerate}[label={\upshape (\alph*)}, leftmargin=2pc]}{\end{enumerate}}
\newenvironment{enumerate-(a)-r}{\begin{enumerate}[label={\upshape (\alph*)}, leftmargin=2pc,resume]}{\end{enumerate}}
\newenvironment{enumerate-(a)-5}{\begin{enumerate}[label={\upshape (\alph*)}, leftmargin=2pc,start=5]}{\end{enumerate}}
\newenvironment{enumerate-(A)}{\begin{enumerate}[label={\upshape (\Alph*)}, leftmargin=2pc]}{\end{enumerate}}
\newenvironment{enumerate-(A)-r}{\begin{enumerate}[label={\upshape (\Alph*)}, leftmargin=2pc,resume]}{\end{enumerate}}
\newenvironment{enumerate-(i)}{\begin{enumerate}[label={\upshape (\roman*)}, leftmargin=2pc]}{\end{enumerate}}
\newenvironment{enumerate-(i)-r}{\begin{enumerate}[label={\upshape (\roman*)}, leftmargin=2pc,resume]}{\end{enumerate}}
\newenvironment{enumerate-(I)}{\begin{enumerate}[label={\upshape (\Roman*)}, leftmargin=2pc]}{\end{enumerate}}
\newenvironment{enumerate-(I)-r}{\begin{enumerate}[label={\upshape (\Roman*)}, leftmargin=2pc,resume]}{\end{enumerate}}
\newenvironment{enumerate-(1)}{\begin{enumerate}[label={\upshape (\arabic*)}, leftmargin=2pc]}{\end{enumerate}}
\newenvironment{enumerate-(1)-r}{\begin{enumerate}[label={\upshape (\arabic*)}, leftmargin=2pc,resume]}{\end{enumerate}}
\newenvironment{itemizenew}{\begin{itemize}[leftmargin=2pc]}{\end{itemize}}

\newtheorem{theorem}{Theorem}[section]
\newtheorem{lemma}[theorem]{Lemma}
\newtheorem{corollary}[theorem]{Corollary}
\newtheorem{proposition}[theorem]{Proposition}

\theoremstyle{definition}
\newtheorem{definition}[theorem]{Definition}

\newtheorem{example}[theorem]{Example}

\theoremstyle{remark}
\newtheorem{remark}[theorem]{Remark}

\begin{document}

\title[]{Simplicity of the automorphism groups of order and tournament expansions of homogeneous structures}

\date{}
\author[F.~Calderoni]{Filippo Calderoni}
\address{Department of Mathematics, Statistics, and Computer Science,
University of Illinois at Chicago, Chicago (IL) 60613, USA}
\email{fcaldero@uic.edu}

\author[A.~Kwiatkowska]{Aleksandra Kwiatkowska}
\address{Institut f\"ur Mathematische Logik und Grundlagenforschung, Westfalische Wilhelms-Universit\"at M\"unster, Einsteinstr. 62, 48149 M\"unster,
Germany \textbf{and} Instytut Matematyczny, Uniwersytet
Wroc\l{}awski, pl. Grunwaldzki 2/4, 50-384 Wroc\l{}aw, Poland}
\email{kwiatkoa@uni-muenster.de}

\author[K.~Tent]{Katrin Tent}
\address{Institut f\"ur Mathematische Logik und Grundlagenforschung, Westfalische Wilhelms-Universit\"at M\"unster, Einsteinstr. 62, 48149 M\"unster,
Germany}
\email{tent@wwu.de}

\thanks{The authors were funded by the DFG under the Germany's Excellence Strategy ---EXC 2044--- Mathematics M\"unster: Dynamics--Geometry--Structure. The second author was also supported by Narodowe Centrum Nauki grant 2016/23/D/ST1/01097. }

\begin{abstract} 
We define the notions of  a free fusion of structures and a weakly stationary independence relation and use them to prove simplicity for the automorphism groups of order and tournament expansions of homogeneous structures like the {rational} bounded Urysohn space, the random graph, and the random poset.
\end{abstract}

\maketitle

\section{Introduction}

This article contributes to the study of the automorphism groups of countable structures.
Such groups are natural examples of separable and completely metrisable topological groups. The richness of their topological properties have recently brought to light a crucial interplay between Fra\"iss\'e amalgamation theory and other areas of mathematics like topological dynamics, Ramsey theory, and ergodic theory. (See \cite{AngKecLyo} and \cite{KecPesTod}.)

The program of understanding the normal subgroup structure of these groups dates back at least to the '50s, when
Higman~\cite{Hig54} proved that \(\Aut(\QQ,<)\), the group of order-preserving permutations of the rational numbers, has very few normal subgroups.\footnote{In fact, the only nontrivial normal subgroups of \(\Aut(\QQ,<)\) are the obvious ones: the one consisting of those automorphisms that fix point-wise some interval \((a, \infty)\), the one consisting of those automorphisms that fix point-wise some interval \((- \infty, b)\), and their intersection.}

In recent years, Macpherson and Tent~\cite{MacTen} proved simplicity for a large collection of groups that arise in a similar fashion as automorphism groups of homogeneous structures.
Their methods encompass a number of examples that had been considered before by various authors: the random graph~\cite{Tru85}, the random  \(K_{n}\)-free graphs and the random tournament~\cite{Rub}, 
and many others.
However, as the authors of~\cite{MacTen} pointed out, their framework does not apply to ordered or even partially ordered structures, in particular it does not apply to the random poset whose automorphism group was proved to be simple in  \cite{GMR}.

A few years later, Tent and Ziegler~\cite{TenZie} introduced the notion of  a \emph{stationary independence relation} and investigated automorphism groups of structures allowing for such a relation.
Their approach is very general: apart from recovering the cases from \cite{MacTen} it
 applies to the random poset and 
 many homogeneous metric structures like the Urysohn space and its variations.
However,  ordered homogeneous structures like the ordered random graph and the random tournament do not carry such a stationary independence relation.

In this article we weaken the notion of a stationary independence relation from~\cite{TenZie} to study the automorphism groups of many order and tournament expansions of structures arising naturally in Fra\"iss\'e amalgamation theory. We believe that such \emph{weakly stationary independence relations} will also be useful in other expansions of homogeneous structures.

Before stating our main theorem, we introduce some terminology. 

\begin{definition}\label{def:fusion}
Let $L_i, i=1,2,$ be  disjoint relational languages and let ${\bf M}_i, i=1,2,$ 
countable homogeneous $L_i$ structures on the same universe $M$. We call an $\mathcal{L}^{*}= L_1\cup L_2$ structure ${\bf M}^*$ on $M$ the free fusion of ${\bf M}_1$ and ${\bf M}_2$ if ${\bf M}^*\restriction L_i =M_i$, $i=1,2,$ and
\begin{itemizenew}
\item [\((\ast)\)] for every non-algebraic $L_i$-type $p_i$ over a finite set $A_i\subset M$ for  $i=1, 2,$ their union $p_1\cup p_2$ is realized in ${\bf M}^*$.
\end{itemizenew} 
For any $\mathcal{L}^*$-type $p$ and $L\subset\mathcal{L}^*$, we write $p_L$ for its restriction to $L$.
\end{definition}

All types considered in this article are types over finite sets.

Similar constructions were considered by {Bodirsky \cite{Bod}} and Soki\'c in \cite{Sok}.

We are particularly interested in the following special cases:

\noindent
{\bf I. Order expansion}\quad
Let \(L_1\) be a relational language and \(\mathbf{M} =\mathbf{M}_1\) 
be a countable homogeneous \(L_1\)-structure on a set $M$.
Let $L_2=\{<\}$ and $\mathbf{M}_2\cong \QQ$ be a dense linear ordering on $M$.
In this case we denote the free fusion of $\mathbf{M}_1$ and $\mathbf{M}_2$  by  \(\mathbf{M}_<\) and call it an \emph{order expansion} of \( \mathbf{M}\).
Thus, an \(\mathcal{L}^{*}\)-structure \(\mathbf{M}_<\) is
an order expansion of \( \mathbf{M}\)
if \(<\) is a total order on \(\mathbf{M}\) {without endpoints}, \(\mathbf{M}_< \restriction L_1 = \mathbf{M}\), and \(\mathbf{M}_{<}\) satisfies the following property:

\begin{itemizenew}
\item [\((\ast)\)] For every non-algebraic $1$-type \(p_{L_1}\) over a finite set \(A\), and every interval \((a,b)\subseteq M\), there is a realization of \(p\) in \((a,b)\).
\end{itemizenew}

 \noindent
{\bf II. Tournament expansion}\quad
Let \(L_1\) be a relational language and \(\mathbf{M}=\mathbf{M}_1\) be a homogeneous \(L_1\)-structure on a set $M$.
Let $L_2=\{\rightarrow\}$ and $\mathbf{M}_2$ be a random tournament on $M$.
In this case we denote the free fusion of $\mathbf{M}_1$ and $\mathbf{M}_2$  by  \(\mathbf{M}_\rightarrow\) and call it a \emph{tournament expansion} of \( \mathbf{M}\).
Thus, an \(\mathcal{L}^{*}\)-structure \(\mathbf{M}_\rightarrow\) is
a tournament  expansion of \( \mathbf{M}\)
if \(\rightarrow\) is a tournament on \(M\), \(\mathbf{M}_\rightarrow \restriction L_1 = \mathbf{M}\), and \(\mathbf{M}_\rightarrow\) satisfies the following property:

\begin{itemizenew}
\item [\((\ast)\)] For every non-algebraic $1$-type \(p_{L_{1}}\) over a finite set \(X\), and two disjoint finite subsets $A, B \subseteq M$, there is a realization $x$ of \(p\) such that $x\rightarrow a$ for all $a\in A$ and $b\rightarrow x$ for all $b\in B$.
\end{itemizenew}

\begin{remark}\label{rem:ordered Fraisse}

Note that if ${\bf M}_i, i=1,2,$ is the Fra\"iss\'e limit of some $L_i$-class $\mathcal{C}_i, i=1,2$ having disjoint amalgamation, then a structure ${\bf M}^*$ 
is the free fusion of ${\bf M}_1$ and ${\bf M}_2$  if and only if  ${\bf M}^*$ 
is the Fra\"iss\'e limit of the $\mathcal{L}^*$-class $\mathcal{C}^*$ where an $\mathcal{L}^*$-structure $A^*$ is in $\mathcal{C}^*$ if and only if $A^*\restriction L_i \in\mathcal{C}_i, i=1, 2$.

Thus, a structure ${\bf M}^*$ 
is an order expansion of a Fra\"iss\'e limit ${\bf M}$ with disjoint amalgamation if and only if  ${\bf M}^*$ 
is the Fra\"iss\'e limit of the class $\mathcal{C}_<$ where an $\mathcal{L}^*$-structure $A_<$ is in $\mathcal{C}_<$ if and only if $A_<\restriction L_1 \in\mathcal{C}$. Equivalently,  $\mathcal{C}_<$ consists of all $A\in\mathcal{C}$ expanded by all possible orderings.
Similarly for the tournament expansion of a Fra\"iss\'e limit.
\end{remark}
Our main theorem can now be stated as follows:

\begin{theorem}
\label{thm:main}
Assume that  \(\mathbf{M}\) is one of the following:
\begin{enumerate}
\item the Fra\"iss\'e limit of a free, transitive and nontrivial amalgamation class\footnote{We call an amalgamation class nontrivial, if its limit is not an indiscernible set and transitive if the automorphism group of its Fra\"iss\'e limit is transitive.}; 
\item the bounded rational Urysohn space; or
\item the random poset.
\end{enumerate}

If \(\mathbf{M}^*\) is an order expansion  of \(\mathbf{M}\),
then \(G:=\Aut(\mathbf{M}^*)\) is simple.
The same holds if\/ $\mathbf{M}^*$ is  a tournament expansion of $(1)$ or $(2)$.
\end{theorem}


Apart from the {ordered bounded rational Urysohn space} see (e.g. \cite{TenZiebounded}), Theorem~\ref{thm:main} implies simplicity of the automorphism groups of various countable structures including 
the ordered random poset, 
the ordered random graph, the ordered random {$n$-hypergraphs}, the ordered  {random} \(K_{n}\)-free graphs and their hypergraph analogues. 

\begin{remark}
In the same way as in \cite{TenZie} we can also conclude that for the ordered rational Urysohn space $\mathbb{U_<}$, the quotient of $\Aut(\mathbb{U_<})$ modulo the normal subgroup of automorphisms of bounded displacement is a simple group.

Note also that with minor modifications the same proof applies to expansions by several independent tournaments or indeed for any binary homogeneous structure with the property that the union of $1$-types over disjoint finite sets is always consistent.
Furthermore, it is worth pointing out that the proof also adapts to several independent orderings. These generalisations were shown by Silke Mei\ss ner in her master's thesis \cite{Mei}.

\end{remark}

Theorem~\ref{thm:main} is proved in two main steps:  we first define a notion of moving maximally adapted to free fusion structures and a notion of compatibility  and prove:

\begin{theorem}
\label{thm:4conjugates}
Let ${\bf M}^*$ be the free fusion of a homogeneous $L_1$-structure ${\bf M}_1$ carrying a stationary independence relation with an $L_2$-structure ${\bf M}_2$. If $g\in G$ moves maximally and is compatible, then any element of $G$ is the product of at most eight conjugates of $g$ and  $g^{-1}$.
\end{theorem}

In Section~\ref{Sec:MoveMaximally} we  prove simplicity of $G$ for the ordered random posets and 
complete the proof of Theorem~\ref{thm:main} by proving:

\begin{proposition}\label{prop:ncl}
If\/ ${\bf M}^*$ is an order or tournament expansion of a structure ${\bf M}$  as in Theorem~\ref{thm:main} (1) or (2)  and ${\id\neq h}\in G=\Aut({\bf M}^*)$, then there is some $g\in \langle h\rangle^G$ that moves maximally and is compatible.
\end{proposition}

Clearly, Proposition~\ref{prop:ncl} and Theorem~\ref{thm:4conjugates} imply Theorem~\ref{thm:main} for the cases (1) and (2).

 
\section{Background and definitions}


First we recall the definition of a stationary independence relation due to Tent and Ziegler.
We also follow their convention by saying that a tuple $a'$ from a countable structure $M$ realizes a type $\tp(a/A)$ over a finite set $A\subset M$ if there is an automorphism of  \(\mathbf{M}\) that maps $a$ to $a'$
and fixes $A$ pointwise.

\begin{definition}{\cite[Definition~2.1]{TenZie}}
Let \(\mathbf{M}\) be a countable structure with universe \(M\) and 
let \(\indep[]\) be a ternary relation between finite subset of \(M\). 
We say that \(\indep[]\) is a \emph{stationary independence relation} on ${\bf M}$ if for all finite sets \(A,B,C,D \subseteq M\) the following hold:
\begin{enumerate-(i)}
\item
(Invariance)
The independence of \(A\) and \(B\)  over \(C\) depends only on the type of \(ABC\). In other words, if $\tp(ABC)=\tp(A'B'C')$, then \[{A}\indep[C]{B} \text{\quad  if and only if  \quad } {A'}\indep[C']{B'}.\]
\item
(Monotonicity)
\(A\indep[B] CD\) implies that \(A\indep[B]C\) and \(A\indep[BC]D\). 
\item
{\rm(Transitivity)
\footnote{As noted by several people  Transitivity follows from the other axioms. We include it here for convenience.}}
    \[{A}\indep[B]{C}\text{ and }{A}\indep[BC]{D}\;
    \text{ implies }\;{A}\indep[B]{{C}D}.\]
\item
(Symmetry) \(A\indep[B]C\) if and only if \(C\indep[B] A\).
\item
(Existence)  If \(p\) is a non-algebraic type over \(B\) and \(C\) is a finite set, there is some \(a\) realizing \(p\) such that \(a\indep[B]C\).
\item
(Stationarity) If the tuples \(x\) and \(y\) have the same type over \(B\) and are both independent from \(C\) over \(B\), then \(x\) and \(y\) have the same type over \(BC\).
\end{enumerate-(i)}
\end{definition}

{Moreover, we will write $A\indep[]C$ for $A\indep[\emptyset]C$.}
\begin{remark}\label{rem:1-type induction} As in \cite[2.4]{TenZie}
we let  \(A\indep[C;D]B\) denote the conjunction of \(AC\indep[D]B\) and \(A\indep[C]BD\).
Then  Transitivity and Monotonicity imply
\[x_1x_2\indep[A;B]y_1y_2\mbox{ \ if and only if \ }\left[x_1\indep[A;B]y_1 \mbox{ and } x_2\indep[Ax_1;By_1]y_2 \right].\]
\end{remark}

\begin{remark}
Recall from \cite{TenZie} that if \(\mathbf{M}\) is the limit of a Fra\"iss\'e class of structures with free amalgamation, then \(\mathbf{M}\) admits a stationary independence relation:
define \(A\indep[B]C\) if and only if $ABC$ is isomorphic to the free amalgam of $A$ and $C$ over $B$, i.e. if and only if $A\cap C=B$ and for every \(n\)-ary relation \(R\) in \(L\), if \(d_{1},\dotsc, d_{n}\) is an \(n\)-tuple in \(A\cup B\cup C\) with some \(d_{i}\in A\setminus B\) and \(d_{j}\in C\setminus B\), then \(R(d_{1}, \dotsc, d_{n})\) does not hold.
 (See~\cite[Example~2.2]{TenZie}.)

For the stationary independence relation on the (bounded) Urysohn space and other metric spaces, we put $A\indep[C]B$ if and only  if for all $a\in A, b\in B$  there is some $c\in C$ such that $d(a, b)=d(a,c)+d(c,b)$, and $A\indep[]B$ if and only  if for all $a\in A, b\in B$  the distance $d(a,b)$ is maximal, see \cite{TenZie}.

In the same vein, the random poset carries a natural stationary independence relation, namely $A\indep[C]B$ if and only if $A\cap B\subset C$ and for all $a\in A\setminus C, b\in B\setminus C$ such that $a\po b$ or $b\po a$ there is some $c\in C$ such that $a\po c \po b$ or $b\po c\po a $, respectively, where we write $\po$ for the partial order of the random poset (see \cite[4.2.1]{KapSim}). Note that we have  $A\indep[]B$ if and only  if no element of $a$ is comparable in the partial order to any element of $B$, i.e. if for all $a\in A, b\in B$ we have $a\npo b$ and $b\npo a$.
\end{remark}

\begin{definition}\label{def:weak}
Let ${\bf M}^*$  be the free fusion of countable $L_i$-structures  ${\bf M}_i, i=1,2,$ with universe \(M\) and 
let \(\indep[]\) be a ternary relation between finite subset of \(M\). 
We say that $\indep$ is a \emph{weakly stationary independence relation} on  ${\bf M}^*$
if it satisfies \rm{(Invariance), (Monotonicity), (Symmetry), (Existence)} and 
\begin{itemize}
\item[{(v')}]
{\rm (Weak Stationarity)} If  \(x\) and \(y\) have the same $\mathcal{L}^*$-type over \(B\) and are both independent from \(C\) over \(B\), then \(x\) and \(y\) have the same $L_1$-type over \(BC\). Thus, if furthermore $\tp_{L_2}(x/BC)=\tp_{L_2}(y/BC)$, then $x$ and $y$ have the same $\mathcal{L}^*$-type over $BC$.
\end{itemize}
\end{definition}

We first note the following:

\begin{proposition}\label{prop:weak stationary independence}
Let ${\bf M}^*$ be the free fusion of a homogeneous $L_1$-structure ${\bf M}_1$ carrying a stationary independence relation $\indep$ with an $L_2$-structure ${\bf M}_2$. Then on ${\bf M}^*$ the relation $\indep$ is a \emph{weakly stationary independence relation}.
\end{proposition}
\begin{proof}
All properties except (Existence) follow immediately.
To see that (Existence) holds for $\mathcal{L}^*$-types, let $p$ be an  $\mathcal{L}^*$-type over a finite set $B$ and let $C$ be a finite set. By (Existence) for $L_1$-types there is a realization $a$ of $p_{L_1}$ with $a\indep[B]C$.
By ($*$), there is a realization $b$ of $\tp_{L_1}(a/BC)$ realizing $p_{L_2}$. Then $b$ realizes $p$ and $b\indep[B]C$.
\end{proof}

\begin{remark}\label{rem:indep} The proof shows in fact that for any finite sets $A, B$ and $C$, any $L_1$-type $p_{L_1}$ over a finite set $A$ and any $L_2$-type $q_{L_2}$ over a finite set $B$ there is a realization of $p_{L_1}\cup q_{L_2}$ which is independent from $C$ over $A$.
\end{remark}

We first note the following adaptation from \cite{TenZie}:
\begin{lemma}\label{L:kechris_rosendal}
If\/ ${\bf M}^*$ is the free fusion of a homogeneous $L_1$-structure ${\bf M}_1$ carrying a stationary independence relation $\indep$ with some $L_2$-structure  ${\bf M}_2$. If $\Aut({\bf M}_2)$ has a dense conjugacy class,
then $G=\Aut({\bf M}^*)$ has a dense conjugacy class as well.
\end{lemma}
\begin{proof}
  Clearly, $G$ contains a dense conjugacy class if and
  only if for any finite tuples $\bar x,\bar y,\bar
  a, \bar b$ with $\tp(\bar x)=\tp(\bar y)$ and $\tp(\bar a)=\tp(\bar
  b)$ there are tuples $\bar x', \bar y'$ such that $\tp(\bar x'\bar
  y')=\tp(\bar x\bar y)$ and $\tp(\bar x'\bar a)=\tp(\bar y'\bar
  b)$. Since ${\bf M}^*$ carries a weakly stationary independence relation,  we can
  choose $\bar x'\bar y'$ realising $\tp(\bar x\bar y)$ with ${\bar
    x'\bar y'}\indep[]{\bar a\bar b}$. By condition ($*$) and since $\Aut({\bf M}_2)$ has a dense conjugacy class, we can choose  $\bar x'\bar y'$  so that $\tp_{L_2}(\bar x'\bar a)=\tp_{L_2}(\bar y'\bar
  b)$. By Weak Stationarity we then have $\tp(\bar
  x'\bar a)=\tp(\bar y'\bar b)$.
\end{proof}

\noindent
{\bf Notation}:  For $a,g\in G$, we write $g^a=a^{-1}ga$ and $[a,g]=a^{-1}g^{-1}ag$.
\begin{definition}\label{def:homogeneous}
We say that $g\in G$ is $L_2$-\emph{homogeneous} if 
for any element $x\in M$ and $a\in G$ we have
$\tp_{L_2}(g(x)/x)=\tp_{L_2}(g^a(x)/x)$, or, equivalently, if for all $x, y\in M$ we have $\tp_{L_2}(xg(x))=\tp_{L_2}(yg(y))$.
\end{definition}

\begin{example}
Let $\mathbf{M}^*$ be the free fusion of a homogeneous $L_1$-structure $\mathbf{M}_1$ carrying a stationary independence relation with an $L_2$-structure $\mathbf{M}_2$ and let $G=\Aut(\mathbf{M}^*)$. 

\medskip\noindent
1. If $\mathbf{M}_2$ is the trivial structure, any fixed point free $g\in G$ is $L_2$-homogeneous.

\medskip\noindent
2. If ${\bf M}_<$ is an order expansion of 
${\bf M_{1}}$, then $g\in G$ is $<$-homogeneous if and only if { \(g\) is the identity isomorphism, or} $g$ is strictly increasing or $g$ is strictly decreasing.

\medskip\noindent
3. If ${\bf M}_\rightarrow$ is a tournament expansion of 
{${\bf M_{1}}$}, then $g\in G$ is  $\rightarrow$-homogeneous if and only if \(g\) is the identity isomorphism, or $a\rightarrow g(a)$ for all $a\in M$, or  $g(a)\rightarrow a$ for all $a\in M$.

\medskip\noindent
4. If $\mathbf{M}_2$ is the random graph, then  $g\in G$ is 
$E$-homogeneous if and only if  \(g\) is the identity isomorphism, $E(a,g(a))$ for all $a\in M$ or $\neg E(a,g(a))$ for all $a\in M$. 
\end{example}

\bigskip

\begin{remark}\label{rem:L2 homogeneous automorphism}
Using property $(*)$ and back-and-forth it is easy to construct examples of $L_2$-homogeneous automorphisms in each of the cases above.
\end{remark}

The free fusion of two structures both having a stationary independence relation has again a stationary independence relation. Therefore, while our methods will  transfer, we do not consider e.g.  expansions by graphs in this article.

\begin{definition}\label{def:move max}
Let ${\bf M}^*$ be the free fusion of a homogeneous $L_1$-structure ${\bf M}_1$ carrying a stationary independence relation with an $L_2$-structure $\mathbf{M}_2$ and let $G=\Aut({\bf M}^*)$. 
We say that  $g\in G$ {\em moves maximally} if
\begin{enumerate-(i)}
\item \(g\) is $L_2$-homogeneous; and
\item\label{moved max} 
every non-algebraic type over a finite set $X$ has a realization $x$ such that
\label{eq:MovedMaximally}
\[x\indep[X;g(X)]g(x).\]
\end{enumerate-(i)}
When $x$ is a realization as in \ref{eq:MovedMaximally}, we say that
\(x\) is {\em moved maximally} by \(g\) over $X$.
\end{definition}

Clearly if $g$ moves maximally, then so  do $g^{-1}$ and all conjugates of $g$.

\medskip\noindent

We will frequently use the following refinement of the maximal moving condition:
\begin{proposition}\label{prop:move max in interval}
If $g\in G$ is moving maximally, then for every non-algebraic $n$-type 
$p$ over a finite set $X$ and non-algebraic $L_2$-type $q_{L_2}$ over $A$ such that $q_{L_2}\cup p_{L_2}$ is consistent there is some $y$ realizing $p\cup q_{L_2}$ which is moved maximally by $g$ over $X$. 
\end{proposition}
\begin{proof}
Suppose that $g$ moves maximally and consider a non-algebraic type  $p$ over a finite set $X$ and an
$L_2$-type $q_{L_2}$ over $A\subset M$ such that  $q_{L_2}\cup p_{L_2}$ is consistent. 
Let $y$ be a realization of $p_{L_1}$ such that $y\indep[X]A$. By ($*$) we can choose a realisation $z$ of $\tp_{L_1}(y/XA)\cup q_{L_2}$. Thus $z\indep[X]A$.
Since $g$ moves maximally, there is a realization $c$ of $\tp(z/XA)$ such that $c\indep[XA;g(XA)]g(c)$.
Furthermore, since
$c\indep[XA]g(XA)g(c)$ and $c\indep[X]A$, we get $c\indep[X]g(X)g(c)$ by Transitivity.
Similarly, $g(c)\indep[g(X)]g(A)$ and $cXA\indep[g(XA)]g(c)$ imply $cX\indep[g(X)]g(c)$ 
and hence we see that
$c\indep[X;g(X)]g(c)$.

Thus $c$ realizes $p\cup q_{L_2}$ and is moved maximally by $g$ over $X$.
\end{proof}

\medskip\noindent
{\bf Notation.} For a tuple $x$ we denote the components by $x^i$.

\begin{definition}\label{def:compatible}
Let ${\bf M}^*$ be the free fusion of a homogeneous $L_1$-structure ${\bf M}_1$ carrying a stationary independence relation with an $L_2$-structure $\mathbf{M}_2$ and let $g\in G=\Aut({\bf M}^*)$ move maximally. We say that $g$ is \emph{compatible} if the following holds:

\medskip\noindent
Any finite set $X_0$ has a finite extension $X$ depending only on $X_0$ and $g$ such that the following holds:
 
\medskip\noindent 
(1) for all tuples $x, y$  such that  
$g(\tp(x/X))=\tp(y/Y)$ 
with \(Y=g(X)\),
$\tp_{L_2}(y^i/x^i)=\tp_{L_2}(g(x^i)/x^i)$
and $x\indep[X;Y]y$
there is some $a\in \Fix(XY)$ such that $g^a(x)=y$.
In this case we call $X$ \emph{full} for $g$.

\medskip\noindent 
(2) Suppose that $b\in G$, $Z$ is a finite set such that $g^b(Z)=Y$, $x, z$ are tuples such that $g^b(\tp(z/Z))= g(\tp(x/X))$,  $z\indep[Z]YX$, $x\indep[X]YZ$. Then there is  some $b'\in\Fix(YZ)$, and some finite extension $Z^+$ of $Z$ which is full for $g^{b}$ and such that with $Y^+=g^{b}(Z^+)$  we have $\tp(g^{bb'}(z)/Y)=\tp(g^{b}(z)/Y)$, 
 \[g^{bb'}(z)\indep[Y](b')^{-1}(Y^+)\quad\mbox{and}\quad b'(z)\indep[Z^+]Y^+ \] and the $L_2$-type 
\[p=\tp_{L_2}(g^{bb'}(z)/b'^{-1}(Y^+))\cup\bigcup_i \tp_{L_2}(g(x^i)/x^i)\cup\bigcup_i \tp_{L_2}(g^{bb'}(z^i)/z^i)\] is consistent.
\end{definition}

Note that our notion  $L_2$-homogeneity only refers to 2-types consisting of a single element and its image. This is suitable only for binary languages. While a similar notion can be defined for higher arities, we leave the details for future investigation.
\begin{proposition}\label{prop:compatible}
Let ${\bf M}^*$ be the free fusion of a homogeneous $L_1$-structure ${\bf M}_1$ carrying a stationary independence relation with an $L_2$-structure $M_2$ and let $g\in G=\Aut({\bf M}^*)$ move maximally. If $g$ is \emph{compatible} and $X$ is full for $g$, $Y=g(X)$, then the following holds:

For any finite set $Z$ such that $g(X)=Y=g^b(Z)$ for some $b\in G$ and all tuples $x, z$ such $\tp(g(x)/Y)=\tp(g^b(z)/Y)$ and $x\indep[X]YZ$ and $XY\indep[Z]z$ there are $a_1\in \Fix(XY), a_2\in \Fix(YZ)$ such that $g^{-ba_2}g^{a_1}(x)=z$.
\end{proposition}
\begin{proof}
Since $g$ is compatible,  
 there is a finite extension $Z^+$ of $Z$ which 
is full for $g^{b}$ and there is $b'\in\Fix(YZ)$  such that with 
$Y^+=g^{b}(Z^+)$,  $Y^*={b'}^{-1}(Y^+)$ and $Z^*={b'}^{-1}(Z^+)$ we have  
{$\tp(g^{bb'}(z)/Y)=\tp(g^{b}(z)/Y)$, }
 \[g^{bb'}(z)\indep_Y Y^* \quad\mbox{and}\quad 
{Y^*}\indep[Z^*]z \]  and the $L_2$-type 
\[p_{L_2}=\tp_{L_2}(g^{bb'}(z)/Y^*)\cup\bigcup_i 
\tp_{L_2}(g(x^i)/x^i)\cup\bigcup_i \tp_{L_2}(g^{bb'}(z^i)/z^i)\] is 
consistent. Note that  $Y^*$ is full for $g^{bb'}$.

Let  $y$ be a realization of ${p_{L_2}}\cup\ \tp_{L_1}(g^{bb'}(z)/Y)$ independent from $xXzY^*Z^*$
over $Y$.  Since $y\indep[Y]Y^*$ and $ \tp_{L_1}(y/Y)=\tp_{L_1}(g^{bb'}(z)/Y)$ and 
$g^{bb'}(z)\indep_Y Y^*$,  Stationarity implies 
$\tp_{L_1}(y/Y^*)=\tp_{L_1}(g^{bb'}(z)/Y^*)$. 
Since $\tp_{L_1}(g^b(z)/Y)=\tp_{L_1}(g^{bb'}(z)/Y)$, 
we have $\tp_{L_1}(y/Y))=\tp_{L_1}(g(x)/Y))$, and by the assumption
$\tp_{L_2}(y^i/x^i)=\tp_{L_2}(g(x^i)/x^i)$ and 
$\tp_{L_2}(y^i/z^i)=\tp_{L_2}(g^{bb'}(z^i)/z^i)$.
Moreover,
\[x\indep[X;Y]y\quad\mbox{and}\quad y\indep[Y^*;Z^*]z.\]
Since {$X$ is full for $g$ and $Y^*$ is full for $g^{bb'}$}, we can find the required $a_1, a_2$.
\end{proof}


\section{Proof of the main result}\label{Sec:main theorem}

In this section we prove Theorem~\ref{thm:4conjugates} using the general strategy of \cite{Las92} and \cite{TenZie}.
Let \(\mathbf{M}^*\) be as in the hypothesis of Theorem~\ref{thm:4conjugates} and \(G= \Aut(\mathbf{M}^*)\). 
For \(A\subseteq M\), let \(\Fix(A)\) denote the \emph{pointwise stabiliser} of \(A\).
For \(A,B\subseteq M\), we write \(AB\) for their union \(A\cup B\).

\begin{lemma}\label{L:zug}
  Let $g\in G$ move maximally, let $X,Y,C$ be finite sets such that
  $g(X)=Y$ and ${X}\indep[Y]{C}$ and let $x$ be a tuple. Then there is some
  $a\in\Fix(XY)$ such that
  \[{g^a(x)}\indep[Y]{C}.\]
\end{lemma}
\begin{proof}
This follows as in \cite[3.5]{TenZie}.
\end{proof}

\begin{proposition}\label{P:zurichtung}
  Consider $g_1,\ldots,g_4\in G$ that move maximally and finite sets $X_0,\ldots,X_4$ such
  that $g_i(X_{i-1})=X_i$. Assume that $g_2$ is compatible. Then for $i=1,\ldots, 4$ there are 
  extensions $Y_i\supset X_i$  and  $a_i\in \Fix(X_{i-1}X_i)$ (with $a_2=a_3=\id$) such that
  \begin{enumerate}\upshape
  \item $Y_1$ is full for $g_2$;
  \item $g_i^{a_i}(Y_{i-1})=Y_i$,
  \item ${Y_0}\indep[Y_1]{Y_2Y_3}$ and ${Y_1Y_2}\indep[Y_3]{Y_4}$.
  \end{enumerate}
\end{proposition}
\begin{proof} 
Put $Y_2'=g_2(X_0X_1)\cup g_3\inv(X_3X_4)$ and
$Y_1'=g_2\inv(Y_2')$. Note that $Y_1'$ contains $X_0X_1$. Now let $Y_1$ be an extension of $Y_1'$ which is full for $g_2$ and put $Y_2=g_2(Y_1), Y_3=g_3(Y_2)$. Note that $Y_3$ contains $X_3X_4$. Put $Y_0'=g_1\inv(Y_1)$ and let $Y_0$ realize $\tp(Y_0'/Y_1)$ independent from $Y_2Y_3$. Let $a_1\in\Fix(Y_1)$ such that $a_1(Y_0)=Y_0'$. Then $g_1^{a_1}(Y_0)=Y_1$.
Now put $Y_4'={g_4(Y_3)}$. Let $Y_4$ realize $\tp(Y_4'/Y_3)$ independent from $Y_1Y_2$ and let $a_4\in\Fix(Y_3)$ such that $a_4(Y_4)=Y_4'$. Then $g_4^{a_4}(Y_3)=Y_4$.
\end{proof}

\begin{proposition}\label{P:vier}
  Let $g_1,\ldots,g_4\in G$  be moving maximally and  assume that 
$g_2$ is compatible and conjugate to $g_3^{-1}$. Let $Y_0,\ldots,Y_4$ be
  finite sets such that $g_i(Y_{i-1})=Y_i$ for $i=1,\ldots, 4$. Assume
  also that
 \[{Y_0}\indep[Y_1]{Y_2Y_3} \quad \mbox{ and}  \quad {Y_1Y_2}\indep[Y_3]{Y_4}\]  and that $Y_1$ is full for $g_2$.
Let
  $x_0$ and $x_4$ be two tuples such that $g_4g_3g_2g_1$ maps
  $\tp(x_0/Y_0)$ to $\tp(x_4/Y_4)$. Then for $i=1,\ldots, 4$, there are
  $a_i\in \Fix(Y_{i-1}Y_i)$ such that
  $g_4^{a_4}\ldots g_1^{a_1}(x_0)=x_4$.
\end{proposition}


\begin{proof}[Proof of Proposition \ref{P:vier}.]
Since $g_1$ and $g_4$ move maximally, using Lemma~\ref{L:zug} we find
  $a_1\in \Fix(Y_0Y_1)$ and $a_4\in \Fix(Y_3Y_4)$ such that for
  $x_1=g_1^{a_1}(x_0)\mbox{ \ and\ \ }x_3=(g_4^{-1})^{a_4}(x_4)$
  we have \[{x_1}\indep[Y_1]{Y_2Y_3}\mbox{
    \ and\ \ }{Y_1Y_2}\indep[Y_3]{x_3}.\]
    Since $g_2$ is compatible and $Y_1$ is full for $g_2$, by Proposition~\ref{prop:compatible} we find the required $a_2\in\Fix(Y_1Y_2), a_3\in\Fix(Y_2Y_3)$ such that $g_3^{a_3}g_2^{a_2}(x_1)=x_3$.
 Thus, \[g_4^{a_4}g_3^{a_3}g_2^{a_2}g_1^{a_1}(x_0)=x_4.\]

\end{proof}

\begin{proposition}\label{prop:nwd}
Let \({g_{1},\dotsc,g_{4}}\in G\) be compatible (hence moves maximally), and assume that 
$g_2$ is conjugate to $g_3^{-1}$. 
Then, for any non-empty open set \({U} \subseteq G^{4}\), there is some non-empty open set \(W \subseteq G\) such that the image \(\phi(U)\) under the map
\[
\phi\colon G^{4}\to G\colon (h_{1},\dotsc,h_{4})\mapsto  g_{4}^{h_{4}} g_{3}^{h_{3}} g_{2}^{h_{2}}g_{1}^{h_{1}}.
\]
is dense in \(W\).
\end{proposition}
\begin{proof}
Using Proposition \ref{P:zurichtung} and compatibility of the $g_i$, the  proposition follows exactly as in \cite[2.13]{TenZie}.
\end{proof}

\begin{theorem}
\label{thm:4conjugates,II}
If \(g \in G\) is compatible, then any element of \(G\) is the product of at most eight conjugates of \(g\) and  \(g^{-1}\).
\end{theorem}
\begin{proof}
We can use Proposition \ref{prop:nwd}  and follow the proof in \cite[2.7]{TenZie}.
\end{proof}


\section{Obtaining compatible automorphisms}
\label{Sec:MoveMaximally}

In this section we prove Proposition~\ref{prop:ncl}
assuming that  \(\mathbf{M}\) is either the random poset,   
the Fra\"iss\'e limit of a nontrivial free amalgamation class or
the {rational} bounded  Urysohn space and that
 \(\mathbf{M}_<\) is an order  expansion of \(\mathbf{M}\). {We obtain the same results for the tournament expansion 
 ${\bf M}_\rightarrow$ of the Fra\"iss\'e limit of a nontrivial free amalgamation class and 
the rational bounded  Urysohn space.} 

We first prove that for tournament expansions any automorphism that moves maximally is automatically compatible.

\begin{proposition}\label{prop:tournament and order compatible}
If\ ${\bf M}_\to$ is  a tournament  expansion of a homogeneous $L_1$-structure ${\bf M}_1$ carrying a stationary independence relation, any $g\in G=\Aut({\bf M}_\to)$ that moves maximally is compatible.
\end{proposition}

We first prove the following lemma which uses the fact that in the random tournament the union of any two $1$-types over disjoint finite sets is consistent.

\begin{lemma}\label{lem:tournament step 1}
Let ${\bf M}_\to$ be  a tournament expansion of a homogeneous $L_1$-structure ${\bf M}_1$ carrying a stationary independence relation and let $g$ be an automorphism moving maximally. Then any finite set $X_0$ has a finite extension $X$  which is full for~$g$.
\end{lemma}

\begin{proof}
Let $g$ move maximally and  let $X_0$ be a finite set.
We will show that any extension $X$ of $X_0$ such that $g\inv(X)\cap g(X)\subset X$ is full for $g$. 
Put \[X=X_0\cup\{x\in M\colon x=g(a) \mbox{ and } g(x)=b \mbox{ for some } a,b\in X_0\}.\] Then $X$ is finite and $g\inv(X)\cap g(X)\subset X$. 
Now assume that $x, y$ are as in (1). If $x_0\subset x$ is contained in $X$, then for the corresponding coordinates $y_0$ of $y$ we have $g^a(x_0)=y_0$ for all $a\in \Fix(XY)$. Since $x\indep[X;Y]y$, we may assume that $x\cap XY=\emptyset$.
Let $x'$ be a realization of $\tp(x/XY)$ such that $g(x'), g\inv(x'), g^2(x')$ and $g^{-2}(x')$ do not intersect $X$ and let $a\in\Fix(XY)$ such that $a(x)=x'$.  Thus replacing $g$ by $g^a$ if necessary, we may assume that $g^2(x)$ and $g^{-2}(x)$ do not intersect $X$. Note that this implies that for any coordinate $x_0$ of $x$ the set $X\cup\{x_0\}$ again satisfies $g\inv(Xx_0)\cap g(Xx_0)\subset X$. This will be used below in the induction.

Now we do induction on the length of $x, y$. First suppose that $x$ and $y$ are single elements.

Since $g\inv(X)\cap g(X)\subset X$, the $L_2$-type $\tp_{L_2}(x/XY)\cup \tp_{L_2}(g\inv(y)/g\inv(XY))$ is consistent. This uses the fact that in the random tournament the union of any two types over disjoint finite sets is consistent.
By Proposition~\ref{prop:move max in interval} we can choose   a realisation $x'$ of $\tp(x/X)\cup \tp_{L_2}(x/XY)\cup \tp_{L_2}(g\inv(y)/g\inv(XY))$ which is moved maximally by
  $g$.
  Since $x'\indep[X]Y$, we have
  $\tp(x'/XY)=\tp(x/XY)$ by weak stationary independence. Choose $a_1\in \Fix(XY)$ with
  $a_1(x)=x'$. Then $g^{a_1}$ moves $x$ maximally over $X$. We have  $\tp_{L_1}(g^{a_1}(x)/xXY)=\tp_{L_1}(y/xXY)$ by Stationarity.
From the $L_2$-homogeneity of $g$, we have   $\tp_{L_2}(g^{a_1}(x)/x)=\tp_{L_2}(y/x)$ and therefore, since tournaments have a binary language, we obtain
$\tp_{L_2}(g^{a_1}(x)/xXY)=\tp_{L_2}(y/xXY)$.  Using weak stationary independence
we conclude that $\tp(g^{a_1}(x)/xXY)=\tp(y/xXY)$. Choose $a_2\in \Fix(xXY)$ with  $a_2(y)=g^{a_1}(x)$. Then $g^{a_1a_2}(x)=y$.

For the induction step assume that the claim is proved for tuples of length $n-1$ and let $x,y$ be tuples of length $n$.
Write $x, y$ as $x_0x'$ and $y_0y'$, respectively,  where $x', y'$ are tuples of length $n-1$.
By the first step of the induction we may assume that $g(x_0)=y_0$ so that $\tp(g(x')/y_0Y)=\tp(y'/y_0Y)$ and $x'\indep[x_0X;y_0Y]y'$.  Since $Xx_0$ is again full, we can apply the induction hypothesis and find
$b\in\Fix(x_0Xy_0Y)$ such that $g^b(x')=y'$.
\end{proof}

We can now prove Proposition~\ref{prop:tournament and order compatible}:
\begin{proof}[{Proof of Proposition~}\ref{prop:tournament and order compatible}]
Let  $x, z, {Z,b}$ be such that $\tp(g(x)/Y)=\tp(g^b(z)/Y)$, $x\indep[X]YZ$ and $XY\indep[Z]z$,
{$g(X)=Y$,  $g^b(Z)=Y$, and $X$ is full for $g$}. Clearly we may assume that $x$ does not intersect $X$  (and so $z$ does not intersect $Z$) and since $x\indep[X]YZ$ we see that $x$ is in fact disjoint from $XYZ$. Let $Z^+$ be the finite extension of $Z$ full for $g^b$  as above, i.e.
\[Z^+=Z\cup\{x\in M\colon x=g^b(c)\quad \mbox{and}\quad  g^b(x)=d\quad\mbox{ for some}\quad c,d\in Z\}\] and put $Y^+=g^b(Z^+)$. Then $Z^+\subset YZ$ and $Y^+\subset YZ$. Hence
$Y^+\indep_{Z^+} z$ and $g^b(z)\indep_Y Y^+$ by Monotonicity and Invariance
 and since $x, z$ are disjoint from $Y^+$. Since $g$ is $\to$-homogeneous, the type
\[\tp_\to(g^b(z)/Y^+)\cup\bigcup_i \tp_\to(g(x^i)/x^i)\cup\bigcup_i \tp_\to(g^{b}(z^i)/z^i)\] 
is consistent as required. (Note that in this case we can choose $b'=\id$.)
\end{proof}

We now turn to the case of order expansions.  We call an element $g\in \Aut(\bf{M}_<)$  {\em unboundedly increasing (resp. decreasing)} if it is increasing (resp. decreasing) and
for every $a<b$ in $\mathbf{M}_{<}$ there is
$m\in\NN$ such that $g^m(a)>b$ (resp.  $g^m(b)<a$).

The following observation will be  helpful later on:

\begin{lemma}\label{l:x+1}
If\: $\mathbf{M}_<$ is an order expansion of a homogeneous structure and $g\in G$ is unboundedly increasing, then we can identify\: ${\bf M}_<$ with the rationals $\QQ$ in such a way that $g(x)=x+1$ for all $x\in {\bf M}_<$.
\end{lemma}
\begin{proof}
Pick some $x_0\in {\bf M}_<$,  identify the interval $[x_0, g(x_0)]\subset  {\bf M}_<$ in an order preserving way with $[0,1]\subset\QQ$  and extend.
\end{proof}
Clearly, in the same way we can identify an unboundedly decreasing function on ${\bf M}_<$ with  $g(x)=x-1$ for all $x\in {\bf M}_<$.

\begin{proposition}\label{prop:order compatible}
If ${\bf M}_<$ is  an order expansion of a homogeneous $L_1$-structure ${\bf M}_1$ carrying a stationary independence relation, any unboundedly increasing $g\in G=\Aut({\bf M}_<)$ that moves maximally is compatible.
\end{proposition}

As a first step towards obtaining full extensions for finite sets we prove the following  weaker version (which nevertheless turns out to be sufficient) because the automorphism $b\in \Fix(Xg(X))$ obtained in the proof below is only an $<$-automorphism of $M$, so does not necessarily respect the $L_1$-structure on $M$:

\begin{lemma}\label{lem:order expansions} 
Let ${\bf M}_<$ be an order expansion of a homogeneous $L$-structure and let $g$ be unboundedly increasing. Then for any finite set $X_0\subset M$ there is a finite set $X$ such that for all $n$-tuples $x, y\in M$ with
$\tp(g(x)/g(X))=\tp(y/g(X))$ and $\tp_<(g(x^i)/x^i)=\tp_<(y^i/x^i), 1\leq i\leq n,$ there is some $<$-automorphism $b\in \Fix(Xg(X))$ of $M$ such that $g^{b}(x)=y$.
\end{lemma}

In fact, the proof of Proposition~\ref{prop:order compatible} will show that  the extension $X$ obtained here is in fact full for $g$.

\begin{proof}
If $X_0=\emptyset$, the conclusion holds.  So let $X_0$ be a nonempty finite set, $X_0'=X_0\cup \{g(x)\}$ for some $x\in X_0$, let $x_{\min}=\min X_0',\ x_{\max}=\max X_0'$. Let \[X=\{g^m(x)\colon x\in X_0, m\in\ZZ \mbox{ and } x_{\min}\leq g^m(x)\leq x_{\max}\}\] and put $Y=g(X)$.

We claim that the conclusion holds for $X$ and $g$.  We may identify $\bf M$ with the rationals in such a way that $X=\{0,\ldots, m\},\quad Y=g(X)=\{k,\ldots, k+m\}$ and $g(x)=x+k$ is a shift by $k$ (see Lemma~\ref{l:x+1}).  
Let $x, y$ be increasing $n$-tuples such that $\tp(g(x)/Y)=\tp(y/Y)$ and $\tp_<(x^i+k/x^i)=\tp_<(y^i/x^i)$. We now  define $b\in\Fix(XY)$ such that $g^b(x)=y$.

 

{Let  $b$ on $[0,k]$ be arbitrary fixing $XY$. Then, in turn, we  extend $b$ to $[k,k+1], \ldots, [k+m-1,k+m]$ in such a way that for all $y^i\in [k+j,k+j+1]$, $j=0,\ldots, m-1$, we have $g^b(x^i)=y^i$. 
We now extend the definition of $b$ from $[0,k+m]$ to $[0,\infty)$ by extending it stepwise to adjacent intervals using the shift function $g(x)=x+k$. }
 
Let $i$ be maximal such that $x^i\in [0,m+k]$. We extend $b$ to  $[m+k, y^i]$. Assume that 
$x^{i+1},\ldots x^q\in [m+k, y^i]$. Then for $j=i+1,\ldots q$ we have an ordered tuple $y^j>y^i$ and corresponding  $g(b(x^j))>g(b(x^{i}))$. So setting $b(y^j)=g(b(x^j))=b(x^j)+k>b(x^i)+k$  for $j=i=1,\ldots q$ preserves the order. Now extend $b$ to all of $[m+k,y^q]$. In the next step, $b$ is already defined on any $x^k\in [y^i,y^q]$ and we continue by defining $b(y^k)$ exactly as before until $b$ is defined on $[0,y^n]$. Then we extend $b$ to $[0,\infty)$. 

{It is left to argue that we can extend $b$ to $(-\infty,0]$.  The argument is essentially symmetric to the one in the previous paragraph (replacing $g$ by $g\inv$ and switching the roles of $x^j$ and $y^j$). }
\end{proof} 
 
\begin{proof}[Proof of Proposition~\ref{prop:order compatible}] 
Identify ${\bf M}$ with the rationals in such a way that the orders agree. Put \(Y = g(X)\). By  Lemma~\ref{lem:order expansions} there is some $<$-automorphism $d\in \Fix(XY)$ such that $y=g^d(x)$. Clearly $\tp_{<}(g^d(x)/xXY)=\tp_{<}(y/xXY)$. It follows that 
$\tp_{<}(g^a(x)/xXY)=\tp_{<}(y/xXY)$ for any $a\in \Fix(XY)$ such that $|a(xy)-d(xy)|= \max\{|a(z)-d(z)|\colon z\in xy\}$ is sufficiently small.
Choose a realisation $x'$ of $\tp(x/X)\cup \tp_{<}(x/XY)$ with $|x'-d(x)|$  sufficiently small which is moved maximally by $g$ over $X$.  
Then as in the corresponding proof  for tournaments we have
  $\tp(x'/XY)=\tp(x/XY)$. Let $a_1\in \Fix(XY)$ with
  $a_1(x)=x'$ and $|a_1(xy)-d(xy)|$  sufficiently small (which is possible by property $(*)$). Thus we have
  $\tp_{<}(g^{a_1}(x)/xXY)=\tp_{<}(y/xXY)$ and  $tp_{L_1}(g^{a_1}(x)/xXY ) = tp_{L_1}(y/xXY )$ by Stationarity.
  Choose $a_2\in \Fix(xXY)$ with  $a_2(y)=g^{a_1}(x)$. Then $g^{a_1a_2}(x)=y$.

Now let  $X$ be as given by Lemma~\ref{lem:order expansions}, $Y=g(X)$, $x, z, Z$, {and $b$}  be such that 
{$g^b(Z)=Y$}, $x\indep[X]{YZ}$ 
and ${XY}\indep[Z]z$,
$\tp(g(x)/Y)=\tp(g^b(z)/Y)$. In particular, $X$ is full for $g$. Let $Z^+$ be a finite  extension for 
$Z$ which is full for $g^b$ defined as above, i.e.
\[Z^+=\{(g^b)^m(x)\colon x\in Z, m\in\ZZ \mbox{ and } z_{\min}\leq 
(g^b)^m(x)\leq z_{\max}\}\] where $z_{\min}, z_{\max}$ are the minimal and the maximal element of $Z$, respectively,
  and put $Y^+=g^b(Z^+)$. 
First note that 
 \[{p_<}=\tp_<(g^b(z)/Y^+)\cup \bigcup_i 
\tp_<(g(x^i)/x^i)\cup\bigcup_i \tp_<(g^{b}(z^i)/z^i)\]
is consistent.
{Indeed, observe first that $y_{\min}\leq Y^+\setminus Y \leq y_{\max}$, where $y_{\min}$ and
$y_{\max}$ are the minimal and maximal element of $Y$. Let also $x_{\min}$ and
$x_{\max}$ be the minimal and maximal element of $X$.}
 If $x^i<x_{\min}$, then $z^i<z_{\min}$ and $g(x^i), g^b(z^i)<y_{\min}$. So we can  find a corresponding coordinate $y^i$ such that $x^i, z^i<y^i<y_{\min}$. 
 Similarly, if $x^i>x_{\max}$, then {$z^i>z_{\max}$ and $g(x^i), g^b(z^i)>y_{\max}$ and we can find a corresponding coordinate $y^i$ such that $y_{\max}<x^i, z^i<y^i$. }
Finally, since $X$ is full for $g$, if $x_{\min}<x^i<x_{\max}$,  
 we have  $x^i<w$ for every  $w$ which realizes $\tp_<(g^b(z^i)/Y^+) $.
Therefore  $y^i=g^b(z^i)$, which clearly  satisfies $\tp_<(g^b(z^i)/Y^+) $ and so $z^i<y^i$, works.


Note that for any $L^*$-automorphism $b'$ sufficiently close to the identity (in the sense of $<$) the type  
\[ p_{b'} = \tp_<(g^{bb'}(z)/{b'}^{-1}(Y^+))\cup \bigcup_i 
\tp_<(g(x^i)/x^i)\cup\bigcup_i \tp_<(g^{bb'}(z^i)/z^i)\]
is also consistent.

Now by Existence and ($*$), let $c$ be a realization of  $\tp_{L_1}(z/YZ)$, sufficiently close to $z$ and such that $ c\indep_{YZ} Y^+Z^+$ and let $b'\in \Fix(YZ)$  with  $b'(z)=c$ be sufficiently close to the identity such that $p_{b'}$ is consistent. {Since $g^b(Z)=Y$ and $\tp_{L_1}(z/Z)=\tp_{L_1}(b'(z)/Z)$, we have 
$\tp_{L_1}(g^b(z)/Y)=\tp_{L_1}(g^b(b'(z))/Y)$, and therefore $\tp(g^{bb'}(z)/Y)=\tp(g^{b}(z)/Y)$.
We claim that $g^{bb'}(z)\indep_Y {b'}^{-1}(Y^+)$, i.e. $c\indep_{Z} Z^+$, and     $c\indep_{Z^+} Y^+$.
Indeed, we have $ c\indep_{YZ} Z^+$
 and $z\indep_Z Y$, and therefore $c\indep_Z Y$. By Transitivity we obtain $c\indep_{Z} Z^+$.}
 Since $c\indep_{YZ} Y^+Z^+$ and $c\indep_Z Y$, by Transitivity we obtain $c\indep_{Z} Y^+Z^+$. Hence by Monotonicity we get $c\indep_{Z^+} Y^+$, as we wanted.
\end{proof}

To prove Proposition~\ref{prop:ncl} for ${\bf M}_<$, we construct an unboundedly increasing automorphism $g\in\langle h\rangle^G$ moving maximally starting from an arbitrary $h\in G$. This is done in four
 steps.
\begin{enumerate-(1)}
\item construct a fixed point free $h_1=[h,f_1]\in\langle h\rangle^G$;
\item construct a  strictly increasing $h_2=[h_1,f_2]\in\langle h_1\rangle^G$;
\item construct an  unboundedly increasing $h_3\in\langle h_2\rangle^G$;
\item construct an unboundedly 
increasing $h_3=[h_2,f_3]\in\langle h_1\rangle^G$ moving maximally.
\end{enumerate-(1)}

We take care of each of these steps in the Lemmas ~\ref{l:no fixed point}, \ref{l:strictly increasing}, \ref{l:unbounded} and \ref{lem:mov max}.

\begin{lemma}{\rm (cf. \cite[3.4(ii)]{MacTen})}\label{lem:orbit}
No element of  \(G \setminus \{\id\}\) fixes any interval pointwise.
\end{lemma}
\begin{proof} Suppose otherwise. Choose $h\in G\setminus\{1\}$ and
 \(a,c,d \in M\)  such that $h$ fixes $[c,d]$ pointwise and $a\neq h(a)$. We first claim that   there exists a finite set  $Y\subset M$  such that $\tp_L(a/Y)\neq \tp_L(h(a)/Y)$ or, equivalently, $\tp_L(Ya)\neq \tp(Yh(a))$. Suppose not. Then by homogeneity of $M$  we can construct an $L$-automorphism of $M$ that fixes $M\setminus\{a, h(a)\}$ pointwise and swaps $a$ and $h(a)$, contradicting \cite[2.10]{MacTen}.

 By property ($*$), $\tp(Y/ah(a))$ is realized by some finite set $B\subseteq (c,d)$. Since $h$ fixes $B$ pointwise, we see that $\tp(h(a)/B) = \tp(a/B)$, a contradiction.

If ${\bf M}_<$ is the ordered rational bounded  Urysohn space, then  the $L_1$-type $p_{L_1}(x)$ expressing $d(x,a)=1\wedge d(x,h(a))=1/2$  is consistent. Using ($*$)  we can pick a realization $b$ of $p_{L_1}(x)$ in the interval $(c,d)$, contradicting the assumption that $h$ is an isometry fixing $b$ and taking $a$ to $h(a)$.

If ${\bf M}_<$ is the ordered random poset, assume $a\npo h(a)$. Using ($*$) we can pick some $x\in (c,d)$ with $x\po a$ and $x\npo h(a)$. Then $\tp(h(a)/x) \neq \tp(a/x)$, contradicting the assumption that $h$ fixes $x$.
\end{proof}

The previous lemma can easily  be adapted to many other fusion structures. However, it seems difficult to give a uniform proof extending the result from \cite{MacTen} to arbitrary fusion structures.

\begin{corollary}\label{cor:fixpointisolate}  
A nontrivial element of $G$ does not fix the set of realizations~$D=p({\bf M}_<)$ of any 
non-algebraic  $\mathcal{L}^{*}$-type  $p$ over a finite set $A$.
\end{corollary}
\begin{proof}
We can assume that $p=q\cup \{a<x<b\}$ for some $a,b\in {\bf M}\cup\{-\infty,\infty\}$ where $q$ is a complete $L$-type over $A$. Let $D'=q(\mathbf{M})$, so $D=D'\cap (a,b)$. By condition \((\ast)\) $D'\cap (a,b)$ is dense in $(a,b)$, so if $h$  fixes $D$ pointwise, then $h$ is the identity on $(a,b)$ and hence  $h=\id$ by Lemma~\ref{lem:orbit}.
\end{proof}

\begin{lemma}\label{l:no fixed point} For any $h\in G$
there is some $g=[h,f]\in \langle h\rangle^G$ which is fixed point free.
\end{lemma}
\begin{proof}
This follows from Corollary~\ref{cor:fixpointisolate} as in \cite[2.11]{MacTen}.  Note that we can use condition ($*$) to ensure that $f$ preserves the ordering.
\end{proof}

\begin{lemma}\label{l:strictly increasing}
If \(h\in G\) has no fixed point, then
there is some $g=[h,f]\in\langle h\rangle^G$ which is  strictly increasing.
\end{lemma}
\begin{proof}
Write \(\mathbf{M}_{<}\) as an ordered union of intervals $J_i, i\in I$, on which $h$ is either strictly increasing or strictly decreasing. By further subdividing we may assume that on each interval $J_i$ the automorphism
$h$ is either unboundedly increasing or unboundedly decreasing.

If $h$ is unboundedly increasing on $J_i$, then as in Lemma~\ref{l:x+1} we identify $J_i$ with (a copy of) $\QQ$
in such a way that 
we have $h(x)=x+1$ for $x\in J_i$. Similarly, if $h$ is unboundedly decreasing on $J_j$, we identify $J_j$ with (a copy of) $\QQ$
in such a way that 
we have $h(x)=x-1$ for $x\in J_j$.

Fix a positive $\epsilon<\frac{1}{2}$.
 We construct an element
$f\in G$ by back and forth  leaving each $J_i$ invariant  and such  that if $h$ is increasing on 
 \(J_{i}\), then we choose $f(x)\in[\frac{x}{2},\frac{x}{2}+\epsilon)$ for $x\in J_i$, or, equivalently, $ f^{-1} (x)\in (2x-2\epsilon,2x]$.  Similarly, if $h$ is decreasing
on 
\(J_{i}\) we choose
 $ f (x)\in[2x,2x+\epsilon)$ for $x\in J_i$, or, equivalently,
$ f^{-1} (x)\in (\frac{x}{2}-\frac{\epsilon}{2},\frac{x}{2}]$.
Clearly, such an $f\in G$ can be easily be constructed by back-and-forth thanks to condition $(*)$.
We then have $h(f(x))>f(h(x))$, and so  $[h,f](x)>x$ for any $x\in \mathbf{M}_{<}$.
\end{proof}

\begin{lemma}\label{l:unbounded}
If $h\in G$ is strictly increasing, there is an unboundedly increasing $g\in \langle h\rangle^G$.
\end{lemma}
\begin{proof}
Since $h$ is strictly increasing, we may identify \(\mathbf{M}_<\) with a  {union} of copies of \(Q\) such that on 
each copy we have \(h(x)=x+1\).

First assume that the  {union} of copies of \(\QQ\)
is infinite in both directions.
Divide each copy of $\QQ$ into an ordered  {union} of two copies \(Q_{1}\cup Q_{2}\) (each \(Q_{i}\) being again isomorphic to \(\QQ\)). 
Using assumption ($*$), we define $f\in G$  by back-and forth such that each 
half-copy of \(Q	\) is moved to the next one above: so \(f(Q_1)=Q_2\) and 
\(f(Q_2)=Q_1\) in the next copy of \(\QQ\).
If there is a first copy $Q_1 \cup Q_2$ we define $f$ in
such a way that \(f(Q_1)=Q_1\cup Q_2\).  And if there is a last copy  $Q=Q_1 
\cup Q_2$ we define $f$ in such a way that $f(Q_1\cup Q_2)=Q_2$.
Then \(g=h\cdot h^f\) is unboundedly increasing.
\end{proof}

\begin{lemma}\label{lem:mov max} If\/ ${\bf M}_<$ is as in Theorem~\ref{thm:main} (1) or (2) 
and \(h\in G\) is unboundedly increasing,
there is some  unboundedly increasing \(g\in \langle h\rangle^G\) that moves maximally.
\end{lemma}

\begin{proof}
Since \(h\) is unboundedly increasing, we may identify  ${\bf M}_<$ with $\QQ$ in such a way that we have \(h(x) = x+1\) for all $x\in{\bf M}_<$.
 Fix a positive \(\epsilon<\frac{1}{2}\).

\medskip 
\noindent 
{\bf Case I:}  If ${\bf M}_<$ is the ordered Fra\"iss\'e limit of a free amalgamation class, we define \(f\) by a back and forth construction like in \cite[5.1]{TenZie}, with the additional requirement that  we have \[
 f(x)\in \Big (\frac{x}{2} - \epsilon, \frac{x}{2} \Big] \mbox{ for each } x\in {\bf M}_<.
 \]
  Since this implies that \(f^{-1}(x)\in [2x, 2x+ {2}\epsilon)\), it follows that
\[
[h,f](x) > x +\delta
\]
with \(\delta = 2 (\frac{1}{2}-\epsilon) > 0\).
Hence the commutator \([h,f]\) will again be   {unboundedly} increasing.

So suppose that \(f'\) is already defined on a finite set \(A\) and let
\(p\) be a type over a finite set \(X\). It suffices to show that \(f'\) has an extension \(f\) such that \([h, f]\) moves \(p\) maximally.

By possibly extending \(f'\) we can assume that \([h,f']\) is defined on \(X\) and that \(f'^{-1}hf'(X)\subseteq A\).
Now pick a realisation \(a\) of \(p\) independent  from \(X' = A\cup h(X)\cup [h,f'](X)\) over \(X\) and  such that $h(a)\neq a$ which is possible by Corollary~\ref{cor:fixpointisolate}. Let \(B= f'(A)\) and pick a realisation \(b\) of \(f'(\tp(a/A))\) in such a way that
\(b\indep_{B}h^{-1}(B)\) and \( b\in (\frac{a}{2} - \epsilon, \frac{a}{2}]\).
Extend \(f'\) to \(Aa\) by setting \(f'(a)=b\).
Next pick a realisation \(c\) of \(f'^{-1}(\tp(h(b)/Bb))\) such that
\(c\) is independent from \(h(a)h(X)\) over \(Aa\), and \(c\in[2h(b), 2h(b)+ {2}\epsilon)\).
Extend \(f'\) by setting \(f'(c) = h(b)\).
Since  weak stationary independence agrees with stationary independence on subsets of ${\bf M}_<$ the proof of Lemma~5.1 in \cite{TenZie} shows that \(a\indep[X;h(X)][h,f'](a)\). 

\medskip
\noindent
{\bf Case II:} Now suppose that ${\bf M}_<$ is the ordered bounded Urysohn space. If there is no $a\in {\bf M}_<$ with $d(a,h(a))=1$, then as in \cite[1.3]{TenZiebounded}  and using condition ($*$) we construct some  {unboundedly} increasing $h_1\in\langle h\rangle^G$ as a product of conjugates of $h$ such that there is some $b\in {\bf M}_<$ with $d(b,h_1(b))=1$: let $0<\epsilon<1$ and $a\in{\bf M}_<$ such that $d(a,h(a))=\epsilon$. Assume $h(a)>a$ (the other case being similar). Pick some $b\in (a,\infty)$
with $d(a,b)=1$. Let $k>1$ be such that $k\epsilon\geq 1$. Put $a_0=a, a_k=b$ and,   using ($*$),  pick  $a_i\in (a_{i-1},b)$ such that $d(a_{i-1},a_i)=\epsilon, i=1,\ldots k$.
Let $f_i\in G$ with $f_i(a_{i-1},a_i)=(a_i,a_{i+1})$ and put $h_1=h^{f_1}\cdot\ldots\cdot h^{f_k}$. 

In the same way we can adapt \cite[2.4]{TenZiebounded} to construct iterated commutators $[h,f]$ using ($*$) to make sure that $f$ preserves  the order and additionally satisfies 
 \[
 f(x)\in \Big (\frac{x}{2} - \epsilon, \frac{x}{2} \Big]
 \] for each \(x\in M\). Thus, after each step, the commutator will again be unboundedly increasing and we end up as in \cite[2.5]{TenZiebounded} with an automorphism $g'\in \langle h\rangle^G$ which is unboundedly increasing and moves \emph{almost maximally}, i.e every nonalgebraic type $p$ over a finite set $X$ has a realization $a$ such that $a\indep[X]g'(a)$. An application of the previous argument as in \cite[5.3]{TenZie} then yields the required $g\in\langle h\rangle^G$ which is  {unboundedly} increasing and moves  maximally.
This concludes the proof in the case of the ordered bounded Urysohn space and thus of Proposition~\ref{prop:ncl}.
\end{proof}

To complete the proof of Theorem~\ref{thm:main} in the case of order expansions, it is left to prove the following two propositions:

\begin{proposition}\label{prop:ordered random poset}
If\/ ${\bf M}^*$ is the ordered random poset, then $G =\Aut({\bf M}^*)$ is simple.
\end{proposition}
\begin{proof}
Let $h\in G$.
By Lemmas~\ref{l:no fixed point},~\ref{l:strictly increasing} and \ref{l:unbounded} we may assume that $h$ is unboundedly increasing in the sense of $<$.
Now we can follow the steps of \cite[Sec. 3]{GMR} to construct some $g\in \langle h\rangle^G$ which is unboundedly increasing in the sense of the partial order $\po$. Using 
property ($*$)  we can make sure that at each step the result is again unboundedly increasing in the sense of the order $<$. It is easy to see using ($*$) that any two elements of $G$ that are unboundedly increasing both in the sense of $\po$ and in the sense of $<$ are conjugate. Adapting the proof of \cite[3.4]{GMR} using ($*$), any element $f$ of $G$ can be written as a product $f=g_1^{-1}g_2$ with $g_1, g_2$ unboundedly increasing in the sense of $\po$ and $<$. Thus $G$
is simple also in this case.
\end{proof}

\begin{proposition}\label{prop:mov max tournament} 
Assume that \(\mathbf{M}\) is the Fra\"iss\'e limit of a nontrivial free amalgamation class 
 {such that $G$ acts transitively on $M$} or
the bounded  rational  Urysohn space and 
 \(\mathbf{M}_\rightarrow\) is a tournament expansion of \(\mathbf{M}\). For any $h\in G$, there is some $g\in\langle h\rangle^G$ moving maximally. 
\end{proposition}
\begin{proof}
As in Lemma~\ref{lem:orbit}, Corollary~\ref{cor:fixpointisolate} and \cite[3.4(ii)]{MacTen} we see that a nontrivial element of $G$ does not fix pointwise the set of realizations of any nonalgebraic type over a finite set. Thus as in Lemma~\ref{l:no fixed point} we can replace $h$ by some fixed point free element in $\langle h \rangle^G$. Then as in Lemma~\ref{lem:mov max}  we can follow the construction of \cite[5.1]{TenZie} and \cite{TenZiebounded} respectively to construct an element $f\in G$ such that  $g=[h,f]$ moves maximally in the sense of the stationary independence relation of \({\bf M}_{1}\). By the axioms of the random tournament and property $(*)$, we can construct $f$ in such a way to ensure that $a\rightarrow [h,f](a)$ for all $a\in M$. Thus, $g=[h,f]$ moves maximally.
\end{proof}

\begin{remark}
An indiscernible set ${\bf M}$ carries a stationary independence relation by setting $A$ to be independent from $B$ over $C$ if $A\cap B\subset C$. Any fixed point free permutation of ${\bf M}$   moves maximally with respect to this stationary independence relation (see \cite{TenZie}). Then ${\bf M}_<$ is isomorphic to $(\QQ,<)$, and we conclude that given any unboundedly increasing automorphism $g$ of $(\QQ,<)$, any $h\in\Aut(\QQ)$ can be written as a product of at most eight conjugates of $g$ and $g^{-1}$.

Similarly, in this case the tournament expansion ${\bf M}_\rightarrow$ is just the countable random tournament and we conclude that for any automorphism $g$ such that $a\rightarrow g(a)$ any $h\in G$ can be written as a product of at most eight conjugates of $g$ and $g^{-1}$.
\end{remark}

Hence Theorem~\ref{thm:4conjugates}  applies to any unboundedly increasing automorphism of $(\QQ,<)$ and to any automorphism $g$ of the random tournament such that $a\rightarrow g(a)$ for all $a$. We thus obtain as a corollary the following result which yields more specific information about Higman's theorem on $\Aut(\QQ,<)$ and the result from \cite{MacTen} on the random tournament $(\TT, \to)$:
\begin{corollary}
If $g$ is an unboundedly increasing automorphism of~$(\QQ,<)$ or an automorphism of the random tournament $(\TT,\to)$  such that $a\rightarrow g(a)$ for all $a\in M$, then any element of $\Aut(\QQ, <)$ or of $\Aut(\TT,\to)$ is the product of at most eight conjugates of $g$ and  $g^{-1}$.
\end{corollary}

\medskip\noindent
{\bf Acknowledgement} The authors would like to thank Binyamin Riahi for alerting
 us to an error in a previous version of the article. We also thank Silke Mei\ss ner for pointing out a number of inaccuracies, in particular in Remark~\ref{rem:ordered Fraisse} where we need the assumption that the Fra\"iss\'e class has disjoint amalgamation.


\begin{thebibliography}{10}

\bibitem{AngKecLyo}
Omer Angel, Alexander S. Kechris, Russell Lyons.
\newblock Random orderings and unique ergodicity of automorphism groups.
\newblock {\em J. European Math. Society} 16 (2014), 2059--2095.

\bibitem{Bod}  M. Bodirsky, 
\newblock New Ramsey classes from old.
\newblock {\em Electron. J. Combin.} 21 (2014), no. 2, Paper 2.22, 13 pp. 


\bibitem{GMR}
 A. M. W. Glass, Stephen H. McCleary, Matatyahu Rubin.
\newblock Automorphism groups of countable highly homogeneous partially ordered sets.
\newblock {\em Math. Z.} 214 (1993), no. 1, 55--66. 

\bibitem{Hig54}
Graham Higman.
\newblock On infinite simple permutation groups.
\newblock {\em Publ. Math. Debrecen}, 3:221--226 (1955), 1954.

\bibitem{KecPesTod}
Alexander~S. Kechris, Vladimir G. Pestov , Stevo Todorcevic.
\newblock  Fra\"iss\'e limits, Ramsey theory,
and topological dynamics of automorphism groups.
\newblock {\em Geom. Funct. Anal.} 15 (2005),
no. 1, 106--189.

\bibitem{KapSim}
Itay Kaplan, Pierre Simon.
\newblock Automorphism groups of finite topological rank.
\newblock  {\em Trans. Amer. Math. Soc.}, 372 (3):2011--2043, 2019.

\bibitem{Las92}
Daniel Lascar.
\newblock Les automorphismes d'un ensemble fortement minimal.
\newblock {\em The Journal of Symbolic Logic}, 57(1):238--251, 1992.

\bibitem{MacTen}
Dugald Macpherson and Katrin Tent.
\newblock Simplicity of some automorphism groups.
\newblock {\em Journal of Algebra}, 342(1):40--52, 2011.

\bibitem{Mei}
Silke Mei\ss ner.
\newblock Order and Tournament Expansions of Homogeneous Structures.
\newblock {Master's thesis,  TU Darmstadt}, 2021.

\bibitem{Rub}
Matatyahu Rubin.
\newblock Unpublished notes.
\newblock 1988.

\bibitem{Sok}
Miodrag Soki\'c. 
\newblock Directed graphs and boron trees. 
\newblock {\em J. Combin. Theory Ser. A} 132 (2015), 142--171. 


\bibitem{TenZie}
Katrin Tent and Martin Ziegler.
\newblock {On the isometry group of the Urysohn space}.
\newblock {\em Journal of the London Mathematical Society}, 87(1):289--303, 11
  2012.
  

\bibitem{TenZiebounded}
Katrin Tent and Martin Ziegler.
\newblock {  The isometry group of the bounded Urysohn space is simple}.
\newblock {\em  Bulletin of the London Mathematical Society}, 45 (2013), no. 5, 1026--1030. 

\bibitem{Tru85}
John~K. Truss.
\newblock The group of the countable universal graph.
\newblock {\em Math. Proc. Cambridge Philos. Soc.}, 98(2):213--245, 1985.

\end{thebibliography}
\end{document}